\documentclass[reqno,11pt,a4paper,final]{amsart}
\usepackage[a4paper,left=40mm,right=40mm,top=43mm,bottom=48mm,marginpar=25mm]{geometry} 
\usepackage{times} 

\usepackage{mathtools}
\usepackage{amsmath}
\usepackage{amssymb}
\usepackage{amsthm}
\usepackage{amscd}
\usepackage[ansinew]{inputenc}
\usepackage{bbm}
\usepackage{color}
\usepackage[english=american]{csquotes}
\usepackage{calc}
\usepackage{bm}
\usepackage{amsmath}
\usepackage{amssymb}
\usepackage{amsthm}
\usepackage[ansinew]{inputenc}
\usepackage{color}
\usepackage[english=american]{csquotes}
\usepackage{calc}
\usepackage{bm}
\usepackage{enumerate}
\usepackage{stmaryrd} 
\usepackage{mathrsfs}
\usepackage{hyperref}

\usepackage{enumerate}

\numberwithin{equation}{section}

\newtheoremstyle{thmlemcorr}{15pt}{15pt}{\itshape}{15pt}{\itshape\bfseries}{.}{5pt}{{\thmname{#1}\thmnumber{ #2}\thmnote{ (#3)}}}
\newtheoremstyle{thmlemcorr*}{15pt}{15pt}{\itshape}{15pt}{\bfseries}{.}\newline{{\thmname{#1}\thmnumber{ #2}\thmnote{ (#3)}}}
\newtheoremstyle{defi}{15pt}{15pt}{}{15pt}{\itshape\bfseries}{.}{5pt}{{\thmname{#1}\thmnumber{ #2}\thmnote{ (#3)}}}
\newtheoremstyle{remexample}{15pt}{15pt}{}{15pt}{\itshape}{.}{5pt}{{\thmname{#1}\thmnumber{ #2}\thmnote{ (#3)}}}
\newtheoremstyle{ass}{10pt}{10pt}{}{}{\bfseries}{.}{10pt}{{\thmname{#1}\thmnumber{ A#2}\thmnote{ (#3)}}}

\theoremstyle{thmlemcorr}
\newtheorem{theorem}{Theorem}[section]
\newtheorem{lemma}[theorem]{Lemma}
\newtheorem{corollary}[theorem]{Corollary}

\theoremstyle{thmlemcorr*}
\newtheorem{theorem*}{Theorem}
\newtheorem{lemma*}[theorem]{Lemma}
\newtheorem{corollary*}[theorem]{Corollary}
\newtheorem{proposition*}[theorem]{Proposition}
\newtheorem{problem*}[theorem]{Problem}
\newtheorem{conjecture*}[theorem]{Conjecture}

\theoremstyle{defi}
\newtheorem{definition}[theorem]{Definition}

\theoremstyle{remexample}
\newtheorem{remark}[theorem]{Remark}

\theoremstyle{ass}

\newcommand{\Crm}{\mathrm{C}}

\newcommand{\Lrm}{\mathrm{L}}

\newcommand{\Wrm}{\mathrm{W}}

\newcommand{\Fcal}{\mathcal{F}}

\newcommand{\Lcal}{\mathcal{L}}

\newcommand{\Pcal}{\mathcal{P}}

\newcommand{\Rcal}{\mathcal{R}}

\renewcommand{\Bbb}{\mathbb{B}}

\newcommand{\Mbb}{\mathbb{M}}
\newcommand{\Nbb}{\mathbb{N}}

\DeclareMathOperator{\esssup}{ess\,inf}

\DeclareMathOperator{\im}{Im}

\DeclareMathOperator{\curl}{curl}
\DeclareMathOperator{\dist}{dist}

\newcommand{\area}[1]{\langle\, #1 \, \rangle}
\newcommand{\barea}{\area{\sbullet}}
\newcommand{\set}[2]{\left\{\, #1 \ \ \textup{\textbf{:}}\ \ #2 \,\right\}}

\newcommand{\setb}[2]{\bigl\{\, #1 \ \ \textup{\textbf{:}}\ \ #2 \,\bigr\}}
\newcommand{\setB}[2]{\Bigl\{\, #1 \ \ \textup{\textbf{:}}\ \ #2 \,\Bigr\}}

\newcommand{\dpr}[1]{\langle #1 \rangle}

\newcommand{\dprB}[1]{\Bigl\langle #1 \Bigr\rangle}

\newcommand{\cl}[1]{\overline{#1}}

\newcommand{\dd}{\;\mathrm{d}}

\newcommand{\R}{\mathbb{R}}

\newcommand{\loc}{\mathrm{loc}}

\newcommand{\toweakstar}{\overset{*}\rightharpoonup}

\newcommand{\toup}{\uparrow}

\newcommand{\BigO}{\mathrm{\textup{O}}}
\newcommand{\SmallO}{\mathrm{\textup{o}}}
\newcommand{\sbullet}{\begin{picture}(1,1)(-0.5,-2)\circle*{2}\end{picture}}
\newcommand{\frarg}{\,\sbullet\,}
\newcommand{\BV}{\mathrm{BV}}
\newcommand{\BD}{\mathrm{BD}}
\newcommand{\LD}{\mathrm{LD}}

\newcommand{\eps}{\epsilon}

\newcommand{\proofstep}[1]{\textit{#1}}

\def\XXint#1#2#3{{\setbox0=\hbox{$#1{#2#3}{\int}$} 
\vcenter{\hbox{$#2#3$}}\kern-.95\wd0}}

\newcommand{\restrict}{\begin{picture}(10,8)\put(2,0){\line(0,1){7}}\put(1.8,0){\line(1,0){7}}\end{picture}}

\renewcommand{\eps}{\varepsilon}
\renewcommand{\phi}{\varphi}
\usepackage{hyperref}
\hypersetup{
    bookmarks=true,         
    unicode=false,          
    pdftoolbar=true,        
    pdfmenubar=true,        
    pdffitwindow=false,     
    pdfstartview={FitH},    
    pdftitle={My title},    
    pdfauthor={Author},     
    pdfsubject={Subject},   
    pdfcreator={Creator},   
    pdfproducer={Producer}, 
    pdfkeywords={keyword1, key2, key3}, 
    pdfnewwindow=true,      
    colorlinks=true,       
    linkcolor=blue,          
    citecolor=blue,        
    filecolor=magenta,      
    urlcolor=cyan           
}
\newcommand{\M}{\mathcal M}

\newcommand{\wc}{\rightharpoonup}

\DeclareMathOperator{\E}{\bf E}

\DeclareMathOperator{\range}{Im}

\DeclareMathOperator{\Ker}{ker}
\DeclareMathOperator{\A}{\mathcal A}

\DeclareMathOperator{\Div}{div}

\author[A.~Arroyo-Rabasa]{Adolfo Arroyo-Rabasa}
\address{\textit{Adolfo Arroyo-Rabasa:} Institut f\"ur Angewandte Mathematik, Universit\"at Bonn, 53115 Bonn, Germany.}
\email{adolfo.arroyo.rabasa@hcm.uni-bonn.de}

\title[Relaxation and optimization for integral functionals under PDE constraints]{Relaxation and optimization for linear-growth convex integral functionals under PDE constraints} 

\begin{document}

\begin{abstract}We give necessary and sufficient conditions for minimality of generalized  minimizers for linear-growth functionals of the form 
		\[
		\mathcal F[u] = \int_\Omega f(x,u(x)) \dd x, \qquad u:\Omega \subset \R^d \to \R^N
		\]
		where $u$ is an integrable function satisfying a general PDE constraint. Our analysis is based on two ideas: a relaxation argument into a subspace of the space of bounded vector-valued Radon measures $\M(\Omega;\R^N)$, and the introduction of a set-valued pairing in $\M(\Omega;\R^N) \times \Lrm^\infty(\Omega;\R^N)$. By these means we are able to  show an intrinsic relation between minimizers of the relaxed problem and maximizers of its dual formulation also known as the saddle-point conditions. In particular, our results can be applied to relaxation and minimization problems in $\BV$, $\BD$.

\vspace{4pt}


\vspace{4pt}

\vspace{4pt}

\noindent\textsc{Date:} \today{}.
\end{abstract}

\maketitle

\section{introduction}

Let $\Omega$ be an open subset of $\R^d$ with $\Lcal^d(\partial \Omega) = 0$. The aim of this work is to establish sufficient and necessary conditions, in the sense of convex duality, for a vector-valued Radon measure $\mu$ to be a  \emph{generalized minimizer} of an integral functional of the form
\begin{equation*}\label{eq:principal}
\mathcal F[u] := \int_{\Omega} f(x, u(x)) \, \text{d}x, 
\end{equation*}
defined on functions $u:\Omega \to \R^N$ 
satisfying the PDE constraint 
\[
\A u = \tau \quad \text{in $\Omega$}, \qquad \text{in the sense of distributions},
\]
where $\A$ is a linear partial differential operator of arbitrary order for which we impose very general conditions (see (H1)-(H2) below)

As part of our main assumptions, $f : \Omega \times \R^N \to [0,\infty)$ is a continuous and convex integrand, that is, $f(x,\frarg)$ is convex for every $x \in \Omega$. We  further assume that $f$ satisfies the following standard coercivity and linear growth assumptions: there exists a positive constant $M$ such that
\begin{equation}\label{eq:linear}
\frac{1}{M}(|z| - 1) \le |f(x,z)|\le M(1 + |z|), \quad \text{for all $(x,z) \in \Omega \times \R^{N}$}.
\end{equation}

We shall consider the linear partial differential operator $\A$ as a linear (possibly unbounded) operator $\A: \Wrm^{\A,1}(\Omega) \subset \Lrm^1(\Omega;\R^N)  \to \Lrm^1(\Omega;\R^n)$, where 
\[
\Wrm^{\A,p}(\Omega) \coloneqq \setB{u \in \Lrm^p(\Omega;\R^N)}{\A u \in \Lrm^p(\Omega;\R^n)}, \quad 1 \le p \le \infty,
\]
is the $p$-integrable $\A$-Sobolev space of $\Omega$.
In this sense, we require that the candidate linear partial differential operator $\A$ satisfies the following assumptions: 
\begin{itemize}
	\item[(H1)] The space
	$\im \A \coloneqq \setb{\A u}{u \in \Wrm^{\A,1}(\Omega)}$ is a closed subspace of $\Lrm^1(\Omega;\R^n)$ with respect to the $\Lrm^1$-strong topology,
	\item[(H2)] for all vector-valued Radon measures $\mu \in \M(\Omega;\R^N)$ such that 
	\[
	\text{$\A \mu = 0$ in $\Omega$, \quad in the sense of distributions},
	\]
	there exists a sequence $(u_j) \subset \Lrm^1(\Omega;\R^N) \cap \ker \A$ such that
	\[
	\text{$u_j \to \mu$ area-strictly in $\M(\Omega;\R^N)$}.
	\]
\end{itemize}
Here, we have denoted by $\M(\Omega;\R^N) \cong (\Crm_b(\Omega;\R^N))^*$ the space of bounded $\R^N$-valued Radon measures in $\Omega$, where we say that a sequence of vector Radon measures $\mu_j$ area-strictly converges to a measure $\mu$ if and only if 
\[
\text{$\mu_j \toweakstar \mu$ in $\M(\Omega;\R^N)$ \quad and \quad $\area{\mu_j}(\Omega) \to \area{\mu}(\Omega)$}, 
\]
for $\barea$ the (generalized) area functional defined in \eqref{eq:areaf}.\\

This general setting contemplates various applications and considerations:
\begin{itemize}
\item[(i)]\emph{On the integrands.} Our results remain the same up to adding a linear term to $f$, i.e., considering integrands of the form
\[
g(x,z) \coloneqq f(x,z) + q^*(x) \cdot z, \quad q^* \in \Lrm^\infty(\Omega;\R^N).
\] 
	\item[(ii)]\emph{On Assumption \rm(H1).} This assumption holds for any operator $\A$ where a  Poincar\'e type inequality holds, for example, if for every $u \in \Wrm^{\A,1}$ it holds that
	\[
	\|u\|_{\Lrm^{1^*}} \le c_\Omega \|\A w\|_{\Lrm^1}, \quad \text{for some $1^* > 1$}.
	\]
	In particular, this includes the relaxation and minimization in $\BV(\Omega;\R^N)$ of problems of the form
	\[
	\int_\Omega f(x,\nabla u(x)) \dd x, \quad u \in \Wrm^{1,1}(\Omega);
	\]
	by taking $\A = \curl$. Similarly, by a Poincar\'e-Korn inequality (see e.g. \cite{Kohn82(KP),Temam83(KP)}),  one can deal with the relaxation and optimization
	in $\BD(\Omega)$ of  problems of the form
	\[
	\int_\Omega f(x,E u(x)) \dd x, \quad u \in \LD(\Omega),
	\]
	where $E u = (D u + (Du)^T)/2$ is the symmetric part of the distributional derivative of $u$,
	and
	\[
	\LD(\Omega) \coloneqq \setB{u \in \Lrm^1(\Omega;\R^d)}{Eu \in \Lrm^1(\Omega;\Mbb^{d \times d}_{\rm{sym}})};
	\]
	 by setting $\A = \curl \curl$ in $\Lrm^1(\Omega;\Mbb^{d \times d}_{\rm{sym}})$. See \cite{BeckSchmidt15} where a generalized pairing in $\BV(\Omega;\R^N) \times \mathrm H_{\Div}(\Omega;\R^N)$ from \cite{Anzellottti83} is used to derive the correspondent saddle-point conditions; 
 see also \cite{KohnTemam82,KohnTemam83} where 
	saddle-point conditions in $\BD$ are established for Hencky plasticity models.
	\item[(iii)]\emph{On Assumption \rm{(H2)}.} If $\A$ is an homogeneous partial differential operator and 
	$\Omega$ is a {\it strictly star-shaped} domain, i.e., there exists $x_0 \in \Omega$ such that 
	\[
	\overline{(\Omega - x_0)} \subset t(\Omega - x_0), \qquad \forall \; t>1;
	\]
	we refer the reader to \cite{Muller87} where such 
	a geometrical assumption is made to address a homogenization problem in the case $\A = \curl$.    
	\item[(iv)] One can also define the problem for more particular domains $D(A)$ than $\Wrm^{\A,1}(\Omega)$ as
	long as $D(A)$ is dense in $\Lrm^1(\Omega;\R^N)$ and (H1)-(H2) hold. In this way one may consider $\A$-Sobolev spaces with zero boundary other characterizations. 
\end{itemize}

\subsection{Main results} 
Our results concern the case where  
\[
\tau = \A u_0\in \Lrm^1(\Omega;\R^n) \cap \im \A,
\] 
for some $u_0 \in \Wrm^{\A,1}(\Omega)$.

Consider the $z$-variable Fenchel  conjugate
$f^* : \Omega \times \R^N \to \cl\R$ of $f$, which is given by the formula\footnote{
	For the sake of simplicity, we depart from the standard notation $(f(x,\frarg))^*$ for the $z$-variable Fenchel transform}
\[f^*(x,z^*) \coloneqq 
\sup_{z \in \R^N}\big\{z^*\cdot z  -  f(x,z)\big\},\qquad z^* \in \R^N.\]

One way to derive optimality conditions for our constrained problem is to consider the dual problems 
\begin{align}\label{eq:problem}
& \inf_{\substack{u \in \Lrm^1(\Omega;\R^N) \\ \A u = \tau \in \im \A}}  \left\{\int_{\Omega} f(x,u(x))\dd x \right\}, && \tag{$\mathcal P$} \\
\label{eq:dualproblem}
& \sup_{\substack{w^* \in \Lrm^\infty(\Omega;\R^n)\\ w^* \in D(\A^*)}}   \left\{\mathcal R[w^*] \coloneqq \dpr{w^*,\tau} -\int_{\Omega} f^*(x,\A^* w^*(x)) \dd x\right\},&& \tag{$\mathcal P^*$} 
\end{align}
where $\A^*$ is the dual operator of $\A$ in the sense of linear (unbounded) operators and where
\begin{align*}
D(\A^*) \coloneqq \setB{& w^* \in \Lrm^\infty(\Omega;\R^n)}{\text{$\exists \; c > 0$ such that} \\ 
 & \text{$|\dpr{w^*,\A u}| \le c\|u\|_{\Lrm^1}$ for all $u \in \Wrm^{\A,1}(\Omega)$}},
\end{align*}
denotes its correspondent domain.
 Using the duality of $\A$ and $\A^*$ 
 it is elementary to check that
	\[
	\mathcal F[u + u_0] \ge \mathcal R[w^*], \qquad \text{for every $u \in \ker \A$ and $w^* \in D(\A^*)$},
	\]
	An immediate observation is that the infimum in \eqref{eq:problem} is greater or equal than the supremum in \eqref{eq:dualproblem}. 
	Convex duality is particularly useful when these two extremal quantities agree since it leads to a saddle-point condition between minimizers of the primal problem and maximizers of the dual problem (we refer the reader to \cite{EkelandTemamBook} for an extensive introduction on this topic). 
	
	Our first result relies on Ekeland's Variational Principle and asserts that there is in fact no gap between these two quantities:
	
	\begin{theorem}\label{prop:nogap} The problems \eqref{eq:problem} and \eqref{eq:dualproblem} are dual of each other and the infimum in problem \eqref{eq:problem} agrees with the supremum in problem \eqref{eq:dualproblem}, i.e., 
		\[
		\inf_{\A u = \tau} \mathcal F[u] = \sup_{w^* \in D(\A^*)} \mathcal R[w^*].
		\]
	Moreover, the supremum in the right hand side  is in fact a maximum, which is equivalent to problem \eqref{eq:dualproblem} having at least one solution.
	\end{theorem} 
	
If a \emph{classical} minimizer $u \in \Lrm^1(\Omega;\R^N)$ of~\eqref{eq:problem} exists and  $w^*$ is a solution of~\eqref{eq:dualproblem}, then the pairing $\langle \A^*  w^*,  u\rangle$ is a saddle-point of these two variational problems. This constitutive relation between $ u$ and $ w^*$ can be derived by standard methods and is expressed by the following pointwise characterization:
	\begin{equation*}\label{eq:calssic}
	f(x,  u(x)) + f^*(x,\A ^*  w^*(x))  =   u(x)  \cdot \A^* w^*(x), \quad \text{for $\Lcal^d$-a.e. $x \in \Omega$}.
	\end{equation*} 

	Under the established assumptions, the infimum of problem~\eqref{eq:problem} is finite and minimizing sequences are $\Lrm^1$-uniformly bounded. 
	It is also well-known (see \cite{Rockafellar,BerLas73,EkelandTemamBook,BouchitteValadier88}) that the convexity of $f(x,\frarg)$ is a sufficient condition to ensure the $\Lrm^1$-weak sequential lower semicontinuity of $\Fcal$, i.e., 
	\[
	\liminf_{j \to \infty} \mathcal F[u_j] \ge \mathcal F[u], \quad \text{whenever $u_j \wc u$ in $\Lrm^1(\Omega;\R^N)$}.
	\]
	However,  
	due to the lack of weak-compactness of $\Lrm^1$-bounded sets, we may only expect that
	\[
	u_j \, \mathcal L^d \toweakstar \mu \in \M(\Omega;\R^N).
	\]
	One must therefore work with the so-called relaxation of $\mathcal F$:
	\begin{theorem}[Relaxation]\label{prop:relax} 
	Let $f : \Omega \times \R^N \to [0,\infty)$ be a continuous integrand with linear growth at infinity as in \eqref{eq:linear}, and such that $f(x,\frarg)$ is convex for all $x \in \Omega$. Further assume that there exists a modulus of continuity $\omega$ such that
	\begin{equation}\label{eq:modulus}
	|f(x,z) - f(y,z)| \le \omega(|x - y|)(1 + |z|) \quad \text{for all $x,y \in \Omega$, $z \in \R^N$}.
	\end{equation}
Then the relaxation, 
\[
\overline{\mathcal F}[\mu] := \setB{\liminf_{j \to \infty} \Fcal[u_j]}{u_j \in u_0 + \ker \A \quad \text{and} \quad u_j \Lcal^d \toweakstar \mu},
\]
of the functional 
	\[
	\Fcal[u] \coloneqq \int_{\Omega} f(x,u(x)) \, \textnormal{d}x, \quad u \in u_0 + \ker \A,
	\]
	is given by the functional
	\begin{align*}
	\mu \mapsto \int_{\Omega} f\left(x, \frac{\dd \mu}{\dd \Lcal^d}(x)\right) \, \textnormal{d}x +  \int_{\Omega} f^\infty\left(x,\frac{\dd \mu^s}{\dd |\mu^s|}(x)\right) \, \textnormal{d}|\mu^s|(x),
	\end{align*}
	defined for measures in the affine space $u_0 + \ker_\M \A$, where
	\[
	 \ker_\M \A:= \setB{\mu \in \M(\Omega;\R^N)}{\text{$\A \mu = 0$ in the sense of distributions}}.
	\]
	Here, $\mu = \frac{\dd \mu}{\dd \Lcal^d} \Lcal^d + \mu^s$ is the Radon-Nikod\'ym decomposition of $\mu$ with respect to $\Lcal^d$ and
	\[
	f^\infty(x,z) \coloneqq \lim_{\substack{x' \to x\\t \to \infty}} \frac{f(x',tz)}{t} \quad (x,z) \in \Omega \times \R^N.
	\]  
	is the {recession function} of $f$.
\end{theorem} 
Extending the differential constraint to $\M(\Omega;\R^N)$, the relaxed functional $\overline{ \mathcal F}$ gives rise to the relaxed problem
\begin{align}\label{eq:rproblem}
\text{minimize $\overline{\mathcal F}$ in the affine space $u_0 + \ker_{\M} \A$}, \tag{$\cl\Pcal$}
\end{align}
for which is possible to guarantee the existence of minimizers.

 Since a (generalized) minimizer~$\overline \mu$ may not be absolutely continuous with respect to $\Lcal^d$, it is not clear 
 in what sense  can ``$\langle \overline \mu, \A^*\overline w^* \rangle$'' be considered a saddle-point of \eqref{eq:rproblem} and \eqref{eq:dualproblem}.	
		 To circumvent the lack of a duality relation in $(\ker_{\M} \A, \im \A^*)$ we introduce a {\it set-valued} pairing as follows:
	\begin{align*}
	\llbracket \mu,\A^* w^* \rrbracket := \setB{\lambda \in \M(\Omega)}{(u_j) \subset & \; u_0 +  \ker \A, \\
		\text{$u_j \to \mu$ area-strictly in $\Omega$, \, and \, } & \text{$(u_j\cdot \A^* w^*) \Lcal^d \toweakstar \lambda$ in $\M(\Omega)$}}.
	\end{align*}

	
We stress that, though our notion of {\it generalized paring} is that of a set-valued pairing, it reduces to a set containing a single Radon measure if stronger regularity assumptions are posed on its arguments $\mu$ or $w^*$. It should also be noticed that the earlier definitions by Anzellotti~\cite{Anzellottti83} for the $(\BV,\Lrm^1 \cap\,\text{$\Div$-free})$ duality, and Kohn and Temam \cite{KohnTemam82,KohnTemam83} in $\BD$ with respect to its dual space, both exploit the {\it potential structure} of gradients and linearized strains, which is not available for the more general constraint $\mu \in \ker_{\M} \A$.
	
	As we will see, it turns out that every $\lambda \in \llbracket \mu , \A^* w^*\rrbracket$ is absolutely continuous with respect to $|\mu|$. 
	Even more, its  absolutely continuous part  with respect to~$\Lcal^d$ is fully determined by $\mu$ and $w^*$ through the relation
	\[
	\frac{\dd \lambda}{\dd\Lcal^d}(x) =\frac{\dd\mu}{\dd\mathcal L^d}(x) \cdot \A^* w^*(x), \qquad \text{for $\mathcal L^d$-a.e. $x \in \Omega$}.
	\]
	This means that, at least formally, elements $\lambda$ in $\llbracket \mu,\A^* w^*\rrbracket$ can be regarded as classical pairings up to a defect singular measure
	$\lambda^s \perp \Lcal^d$.  
	In fact, $\lambda^s$ carries the (generalized) saddle-point conditions as illustrated in our main result:
	\begin{theorem}[Conditions for optimality]\label{thm:optimal} Let $f : \Omega \times \R^N \to [0,\infty)$ be a continuous integrand with linear growth at infinity as in \eqref{eq:linear}, and such that $f(x,\frarg)$ is convex for all $x \in \Omega$. Further assume that there exists a modulus of continuity $\omega$ such that
	\begin{equation}\label{eq:modulus}
	|f(x,z) - f(y,z)| \le \omega(|x - y|)(1 + |z|) \quad \text{for all $x,y \in \Omega$, $z \in \R^N$}.
	\end{equation}
	Then the following conditions are equivalent: 
		\begin{enumerate}
			\item[(i)] $\mu$ is a generalized solution of problem \eqref{eq:problem} and $w^*$ is a solution of \eqref{eq:dualproblem},\vskip5pt
			\item[(ii)] The generalized pairing $\llbracket \mu, \A^* w^*\rrbracket$ is the singleton
			\[
			\left\{ \left(\frac{\dd\mu}{\dd\mathcal L^d} \cdot \A^* w^*\right) \, \mathcal L^d \restrict \Omega ~+ ~ f^\infty\left(\;\frarg \;,\frac{\dd \mu}{\dd |\mu^s|}\right) \, |\mu^s|\right\}.
			\]
			In particular, if $\lambda \in \llbracket \mu,\A^* w^*\rrbracket$, then
			\begin{align*}\label{pointwise}
			\frac{\dd \lambda}{\dd \Lcal^d}(x) & = \frac{\dd\mu}{\dd\Lcal^d}(x) \cdot \A^*w^*(x) \\
			& = f\left(x,\frac{\dd\mu}{\dd\Lcal^d}(x)\right) + f^*(x,\A^*w^*(x))
			\end{align*}
			 for $\mathcal L^d$-a.e. in $x \in \Omega$, and
			 \[
			\frac{\dd \lambda}{\dd |\mu^s|}(x) = f^\infty\bigg(x, \frac{\dd \mu}{\dd |\mu^s|}\bigg) \quad \text{for $|\mu^s|$-a.e. $x \in \Omega$}.
			 \]
		\end{enumerate}
	\end{theorem}
		
	The paper is organized as follows: Firstly, in Section \ref{sec:preliminaries} we give a short account of the properties of integral functionals defined on measures  and their relation to area-strict convergence. The reminder of the Section recalls some facts of convex duality and the commutativity of the supremum on integral functionals for {\it inf-stable} families of measurable functions. 
	In Section \ref{sec:duality} we rigorously derive the dual variational formulation of \eqref{eq:problem} by means of convex analysis arguments and Ekeland's Variational Principle. Section \ref{sec:relaxation} is devoted to the characterization of the relaxed problem \eqref{eq:rproblem}. In Section \ref{sec:pair}, we study the properties of  pairing $\llbracket \mu , \A^*w^*\rrbracket$, from which the proof of Theorem \ref{thm:optimal} easily follows. 

\subsection*{Acknowledgments} The support of the Hausdorff Center of Mathematics and the Bonn International Graduate School is gratefully acknowledged. The results here presented form part of the author's Ph.D. thesis at the University of Bonn.

\section{Ppreliminaries}\label{sec:preliminaries}

\subsection{Notation}
We shall work in $\Omega \subset \R^d$, an open 
domain with $\Lcal^d(\partial \Omega) = 0$ for which we impose no further regularity assumptions. 

By $\Lrm^p_\mu(\Omega;\R^N)$ we denote the subset of $\Lrm_\mu(\Omega;\R^N)$ of $\mu$-measurable functions on $\Omega$ with values in $\R^N$ which are $p$-integrable with respect to a given positive measure $\mu$; we will simply write  $\Lrm_\mu^p(\Omega)$ instead of $\Lrm_\mu^p(\Omega;\R)$, and  $\Lrm^p(\Omega;\R^N)$ instead of 
$\Lrm^p_{\mathcal L^d}(\Omega;\R^N)$), where $\mathcal L^d$ stands for the $d$-dimensional Lebesgue measure. 

In the course of this work we confine ourselves to the use of bounded Radon measures, therefore we will use the notation  $\M(\Omega;\R^N) \cong (\Crm_b(\Omega;\R^N))^*$ to denote the space of $\R^N$-valued Radon measures on $\Omega$ with finite mass. Similarly to $\Lrm^p$, we will simply write $\M(\Omega)$ instead of $\M(\Omega;\R)$. For an arbitrary measure $\mu \in \M(\Omega;\R^N)$ we will often write $\frac{\dd \mu}{\dd \mathcal L^d} \, \mathcal L^d + \mu^s$ to denote its Radon-Nikod\'ym decomposition with respect to $\mathcal L^d$.

We shall write $x \cdot y$ to denote the inner product between two vectors $x,y \in \R^N$. For function and measure spaces, we reserve the notation $\langle \frarg, \frarg \rangle$
to represent the standard pairing between the space and its dual; where no confusion can arise, we shall not emphasize the position of its arguments.
\subsection{Integrands, lower semicontinuity, and area-strict convergence} We recall some well-known and other recent results concerning integrands and recession functions. 

Following \cite{AlibertBouchitte97} and more recently \cite{RindKris10}, we define $\E(\Omega;\R^N)$ as the class of continuous functions $f:\Omega \times \R^N$ such that the transformation
\[
(Sf)(x,z) := (1 - |z|)f\left(x,\frac{z}{1 - |z|}\right) \quad \text{for $(x,z)\in \Omega \times \Bbb^N$}, 
\]
where $\Bbb^N$ is unit open ball in $\R^N$, can be extended to the space $\Crm(\Omega \times \cl{\Bbb^d})$ by some continuous function $\tilde f$. In particular, for every $f \in \E(\Omega;\R^N)$, there exists a positive constant $M > 0$ such that 
\[
|f(x,z)| \le M(1 + |z|) \quad \text{for all $(x,z)\in \Omega \times \R^N$}, 
\]
and
\[
\tilde f(x,z) = \begin{cases}
(Sf)(x,z) & \text{if $|z|< 1$},\\
f^\infty(x,z) & \text{if $|z|= 1$};
\end{cases}
\]
where the limit 
\[
f^\infty(x,z) =  \lim_{\substack{x' \to x\\t \to \infty}} \frac{f(x',tz)}{t} \qquad (x,z) \in  \Omega \times \R^N,
\]
exists and defines a positively $1$-homogeneous function. 

\begin{lemma}[Recession functions I] \label{rem:rec1}
	If $f : \Omega \times \R^N \to \R$ is a continuous convex integrand with linear growth at infinity with a modulus of continuity $\omega$ as in $\eqref{eq:modulus}$,  then $f \in \E(\Omega;\R^N)$. Moreover, the recession function $f^\infty$ exists, is continuous and has the simplified representation
	\[
	f^\infty(x,z) \coloneqq \lim_{t \to \infty} \frac{f(x,tz)}{t}, \quad \text{for all $(x,z) \in \Omega \times \R^N$}.
	\]
\end{lemma}
\begin{proof} First we show that $f(x,\frarg)$ is Lipschitz with $\text{Lip}(f(x,\frarg))\le M$ (independently of $x$). Indeed, by convexity we know that $f(x,\frarg) \in \Wrm^{1,\infty}_\loc(\R^N)$ for all $x \in \Omega$. Fix $x \in \Omega$, then 
\[
\nabla_z f(x,z) \in \partial_z f(x,z) \quad \text{for $\Lcal^N$-almost every $z \in \R^N$}.
\]
Again, by convexity, $p^* \in \partial_z f(x,z)$ if and only if $f^*(p^*) = z \cdot p^* - f(z) \in \R$. Thus, 
$f^*(\nabla f(x,z)) \in \R$ for $\Lcal^N$-almost every $z \in \R^N$. It is easy to check from the linear growth assumption on $f$ that
$\set{p^* \in \R^N}{f^*(p^*) < \infty} \subset M\cdot\Bbb^N$, whereby we deduce that
\[
\|\nabla f(x,\frarg)\|_{\Lrm^\infty} \le M.
\]
This shows that $f(x,\frarg)$ is $x$-uniformly Lipschitz, which together with \eqref{eq:modulus} implies that if $f^\infty$ exists, then
\[
f^\infty(x,z) =  \lim_{t \to \infty} \frac{f(x,tz)}{t}\quad \text{for all $(x,z) \in \Omega\times \R^N$}.
\]
To see that $f^\infty$ exists in $\Omega\times \R^N$ we simply observe that
\begin{equation}\label{eq:nodepende}
\frac{f(x,tz)}{t} = \frac{f(x,tz) - f(x,0)}{t} + \BigO(t^{-1}) \coloneqq I_{x,z}(t) + \BigO(t^{-1}),
\end{equation}
where, by the convexity of $f$, the functions $I_{x,z}(t) \le M$ are monotone for all $(x,z) \in \Omega \times \R^N$.

Finally, to prove that $f \in \E(\Omega ;\R^N)$, we are left to show that $\tilde f$ is continuous at all $(x,z) \in \Omega \times \partial \Bbb^N$ (this, because $f \in \Crm(\Omega \times \R^N)$). Using the modulus of continuity in \eqref{eq:modulus} it is easy to show that $f^\infty$ is continuous, therefore it is enough to show that
\[
\lim_{\substack{x' \to x\\|z'| \toup 1}} \tilde f(x',z) = f^\infty(x,z) \quad \text{for all $x \in \Omega$}.
\]
The latter follows by setting $t(z') \coloneqq \frac{1}{1 - |z'|}$ (which tends to $\infty$ as $|z'| \toup 1$) in \eqref{eq:nodepende}.
\end{proof}

We collect some continuity properties of the class $\E(\Omega;\R^N)$ and recession functions in the following lemmas. The first one is a lower semicontinuity result for convex integrands from \cite{BouchitteValadier88}. The second is a continuity result, originally proved by Reschetnyak in the case of $1$-homogeneous functions \cite{Resh69}, but generalized to lower semicontinuous integrands with linear growth (see for example \cite[Theorem 5]{RindKris10}).

\begin{theorem}
	\label{thm:alibertlsc}Let $(u_j)$ be a bounded sequence in $\Lrm^1(\Omega;\R^N)$ that weakly* converges, in the sense of measures, to a measure $\mu \in \M(\overline \Omega;\R^N)$.  Then 
	\begin{align*}
	\liminf_{j \to \infty} & \int_\Omega f(x,u_j(x)) \dd x \ge  \\ 
	& \int_{\Omega} f\left(x,\frac{\dd \mu}{\dd \Lcal^d}(x)\right) \, \textnormal{d}x + \int_{\Omega} f^\infty\left(x,\frac{\dd \mu^s}{\dd |\mu^s|}(x)\right) \, \textnormal{d}|\mu^s|(x),
	\end{align*}
	for all non-negative integrands $f \in \E(\Omega;\R^N)$. 
\end{theorem}
We introduce the following short notation for the (generalized) area functional of $\mu$,
\begin{equation}\label{eq:areaf}
\langle \mu \rangle(A) := \int_{A} \sqrt{1 + \left|\frac{\dd \mu}{\dd \Lcal^d}(x)\right|^2} \dd x + |\mu^s|(A), \end{equation}
defined on Borel sets $A \subset \R^d$.

We say that a sequence of measures $(\mu_j)$ area-strict converges to $\mu$ if 
\[\mu_j \toweakstar \mu  \quad \text{in $\M(\Omega;\R^N)$} \quad \text{and \quad
$\area{\mu_j}(\Omega) \to \area{\mu}(\Omega)$}.\]
This notion of convergence turns out to be stronger than the usual strict convergence as the latter allows one-dimensional oscillations. The motivation behind the definition of area-strict convergence is that one can formulate the following generalized version of Reschetnyak's Continuity Theorem:

\begin{theorem} The functional  
	\[
	\mu \mapsto \int_{\Omega} f\left(x,\frac{\dd \mu}{\dd \Lcal^d}(x)\right) \, \textnormal{d}x + \int_{\Omega} f^\infty\left(x,\frac{\dd \mu^s}{\dd |\mu^s|}(x)\right) \, \textnormal{d}|\mu^s|(x),
	\]
	is area-strictly continuous in $\Omega$ for every integrand $f \in \E(\Omega;\R^N)$.
	\label{thm:resch} 
\end{theorem}

 \begin{remark}\label{rem:sharpness}It can be easily seen that area-strict convergence is a sharp condition for the continuity of integral functionals defined on measures by taking $f(x,z) := \sqrt{1 + |z|^2} \in \E(\Omega;\R^N)$ and observing that $f^\infty(x,z) = |z|$.
\end{remark}

\subsection{Duality on convex optimization}

We recall some facts of the theory of convex functions. We follow closely those ideas from \cite[Ch. III]{EkelandTemamBook} but we develop them in the slight more general case 
for {\it perturbations with unbounded linear operators}. Along this chapter $X$ and $Y$ will be two topological 
vector spaces placed in duality with their duals $X^*$ and $Y^*$ by the pairing 
$\dpr{\frarg,\frarg}_{X^*\times X}$ (analogously for $Y$ and $Y^*$). The subscript notation
will be dropped as it is understood that the correspondent pairing apply only on their respective domains.
For a continuous function $F : X \to \overline \R$, we define a lower semi-continuous, and convex function by letting
\[
F^*(u^*) := \sup_X \big\{\langle u, u^*\rangle - F(u)\big\}, \qquad u^* \in X^*.
\]
This function is  known as the conjugate function of $F$.
We will be concerned with the minimization problem 
\begin{equation}\label{eq:primalpre}
\text{minimize $F$ in $X$},\tag{p}
\end{equation}
which we term as the {\it primal problem}.

\subsection*{Perturbations} Assume we are given a function 
\begin{align*}
\Phi :  X \times Y \to \overline \R : (u,p) \mapsto \Phi(u,p),
\end{align*}such that
\begin{equation}\label{perturbation 0}
\Phi(u,0) = F(u), \qquad ~\text{for all $u\in X$}.
\end{equation}
The dependence on $p$ is commonly understood as the {\it perturbed problem}.

\subsubsection{The dual problem} Let $\Phi^* : X^* \times Y^* \to \overline \R$ be the conjugate of $\Phi$ in the 
duality $(X\times Y,X^* \times Y^*)$. We define the dual problem of \eqref{eq:primalpre} with respect to the perturbation $\Phi$ as
\begin{equation}\label{eq:dualpre}
\text{maximize $\big\{p^* \mapsto - \Phi^*(0,p^*)\big\}$ in $Y^*$}. \tag{$\text{p}^*$}
\end{equation}
We also define a function, known as the Lagrangian, $L: X \times Y^* \to \overline \R$ by setting
\begin{align*}\label{lagrangian}
-L(u,p^*) & := \big(\Phi(u,\frarg)\big)^*(p^*) \\
& = \sup_{p \in Y} \big\{ \langle p^*,p\rangle - \Phi(u,p)\big\}.
\end{align*}
Hence, 
\begin{equation*}
- \Phi^*(0,p^*) = \inf_{u \in X} L(u,p^*),
\end{equation*}
and if  additionally the function $p\mapsto \Phi(u,p)$ is convex and l.s.c in $Y$, then
\begin{equation}\label{observation}
\Phi(u,0) = \sup_{p^* \in Y^*} L(u,p^*).
\end{equation}
In this case, we observe that
\begin{align*}
\sup \eqref{eq:dualpre} & = \sup_{p^* \in Y^*}\inf_{u \in X} L(u,p^*),\\
\inf \eqref{eq:primalpre} & = \inf_{u \in X}\sup_{p^* \in Y^*} L(u,p^*).
\end{align*}

\subsection{Inf-stability}

Next, we recall some facts on the commutativity of the supremum of integral functionals valued on a certain family 
$\mathscr F$ of measurable functions. The definitions and results gathered here can be found in~\cite[Theorem 1]{BouchitteValadier88}.
During this chapter $\mu$ will denote an arbitrary positive measure. 
\begin{definition}
	A set $\mathscr F$ of $\Lrm_\mu(\Omega;\R^N)$ is said to be \emph{inf-stable} if for any continuous partition of unity 
	$(\alpha_1,...,\alpha_m)$ such that $\alpha_1,...,\alpha_m \in \Crm(\overline \Omega)$, for every $u_1,...,u_m$ in $\mathscr F$,
	the sum $\sum_{i = 1}^m \alpha_iu_i$ belongs to $\mathscr F$. A subset $\mathscr F$ of $\Lrm_\mu(\Omega)$ is called 
	\emph{$\Crm^1$ {inf-stable}} family if for every partition of unity $(\alpha_1,...,\alpha_m)$ such that 
	$\alpha_1,...,\alpha_m \in C^1(\overline \Omega)$ there exists $u \in \mathscr F$ such that $u \le \sum_{i = 1}^m 
	\alpha_iu_i$.
\end{definition}

\begin{theorem}
	For any subset $\mathscr F$ of $\Lrm_\mu(\Omega;\R^N)$ there exists a smallest closed-valued measurable multifunction
	$\Gamma$ such that for all $u \in \mathscr F$, $u(x) \in \Gamma(x)$ $\mu$-a.e. (as smallest refers to inclusion).
	Moreover, there exists a sequence $(u_j)$ in $\mathscr F$ such that $\Gamma(x) = \cl{\{u_j(x)\}}$ for $\mu$-a.e. $x 
	\in \Omega$.
\end{theorem}
	We say that $\Gamma$ is the {\it essential supremum} of the multifunctions 
	\[x \mapsto \setb{u(x)}{u \in \mathscr F},\] in symbols 
	\[
	\displaystyle \Gamma(\frarg) =  \esssup\setb{u(\frarg)}{u \in \mathscr F}
	\]
\begin{theorem}\label{stability}
	Let $j:\Omega \times \R^N \to \overline\R$ be a normal convex integrand. Denote by $J$ the functional 
	\[
 u \mapsto \int_\Omega j(x,u(x)) \dd\mu(x), \qquad \text{for all $u \in \Lrm_\mu(\Omega,\R^N)$}.
	\]
	Let $\mathscr F$ be an \emph{inf-stable} family in $\Lrm_\mu(\Omega,\R^N)$. Assume furthermore that $J$ is proper within $\mathscr F$, 
	i.e., there exists $u_0 \in \mathscr F$ such that $J(u_0) \in \R$. Then, 
	\[
	\inf_{u \in \mathscr F} J(u) = \int_\Omega \inf_{z \in \Gamma(x)} ~ j(x,z) \dd\mu(x),
	\]
	and
	\[
	\inf_{z \in \Gamma(x)} j(\frarg,z) = \esssup \setb{ j(\frarg,u) }{ u \in \mathscr F, J(u) < +\infty} .
	\]
	Moreover, if $\mathscr F$ is a \emph{$C^1$ inf-stable} family of $\Lrm_\mu(\Omega)$, then
	\[
	\inf_{u \in \mathscr F} \int_\Omega u \, d\mu = \int_\Omega \esssup \setb{u(x)}{u \in \mathscr F}  \dd\mu(x). 
	\]
\end{theorem}

\section{The dual problem}\label{sec:duality}

Some of the results of this section are stated under weaker assumptions than the ones 
previously established in the introduction; however, the results in subsequent sections do require stronger these properties  (for a discussion on the sharpness of our assumptions on the integrand $f$ we refer the reader to \cite{BouchitteValadier88} and references therein).

In this section we study the dual formulation \eqref{eq:dualproblem} of \eqref{eq:problem} in the duality 
$(\Lrm^\infty,\Lrm^1)$. Our main goal is to prove Theorem~\ref{prop:nogap} which states not only that \eqref{eq:problem} and \eqref{eq:dualproblem} are in duality but that there is no gap between them. 

The idea is to combine the results of the last section to characterize the dual problem \eqref{eq:dualproblem} as an integral functional in $\Lrm^\infty(\Omega;\R^n)$. Once this is achieved, we will turn to the proof of Theorem \ref{prop:nogap}. Our approach relies on Ekeland's Variational Principle which allows us to work {\it asymptotically close} to the constraint $\A u = \tau$. This  {\it extra} flexibility will be essential to prove that the infimum in problem \eqref{eq:problem} and the supremum in problem \eqref{eq:dualproblem} agree.

For an integrand $g:\Omega \times \R^N \to \R$, we will write $I_g$ to denote  the functional that assigns
\[
u \mapsto \int_\Omega g(x,u(x)) \dd x, \quad u \in \Lrm(\Omega;\R^N),
\]
Following standard notation we denote, for a Banach space $X$ and a subset $U \subset X$, the $U$-indicator function $\delta_X( \frarg | U ) : X \to \cl \R$ defined by the functional 
\[
\delta_X( u \, | U ) \coloneqq \begin{cases}
0 & \hfill \text{if $u \in U$} \\
\infty & \text{if $x \in X \setminus U$}
\end{cases},
\]
which is lower semicontinuous on $\| \frarg\|_X$-closed sets $U \subset X$.
Define $J : \Lrm^1(\Omega;\R^N) \times \Lrm^1(\Omega;\R^n) \to \overline \R$ to be the functional given by
\[
J(u,q) := I_f(u) + \delta_{\Lrm^1(\Omega;\R^n)}(\,q \, | \{\tau\}),
\]
and consider the perturbation with respect to $\A$ given by
\[
\Phi(u,p) := \begin{cases}
J(u,\A u - p) & \quad \text{if $u \in \Wrm^{\A,1}(\Omega)$}\\
\infty& \hfil \text{else}
\end{cases}.
\]
Notice that, the minimization problem \eqref{eq:problem} may be re-written as
\[
\text{minimize  $\Phi(\frarg,0)$ in  $\Lrm^1(\Omega;\R^N)$}. \tag{$\mathcal P$}
\]

\begin{lemma}\label{If} Let $f:\Omega \times \R^N \to \overline \R$ be a continuous and convex integrand with linear growth at infinity. Then the Fenchel conjugate of the functional $I_f : \Lrm^1(\Omega;\R^N) \to \R$, 
 is given by the integral functional
\[
u^* \mapsto I_{f^*}(u^*),
\]
defined on functions $u^* \in \Lrm^\infty(\Omega;\R^N)$.

\end{lemma}
\begin{proof}
We argue as follows.\\
	{\it Step 1.} We point out that $\Lrm^1(\Omega;\R^N)$ is an inf-stable family.\\
	{\it Step 2.} Since $f$ has linear growth, $I_f - \langle u^*, \frarg \rangle$ is proper in $\Lrm^1(\Omega;\R^N)$.\\
	{\it Step 3.} We fix $u^* \in \Lrm^\infty(\Omega;\R^N)$ and apply Theorem \ref{stability} to 
	$\mathscr F = \Lrm^1(\Omega;\R^N)$ (which is an inf-stable family), to $\mu = \dd \mathcal L^d$, and to 
	\[
	j(x,z) = f(x,z)- u^*(x) \cdot z ,\] 
which remains a convex normal integrand, to find out that
	\[
	(I_f)^*(u^*) = -\inf_{u \in \Lrm^1(\Omega;\R^N)} \int_\Omega j(x,u(x)) \dd x  = -\int_\Omega \inf_{z \in \Gamma(x)} j(x,z) \dd x,
	\]
	where 
	$\Gamma(\frarg) =  \esssup\setb{u(\frarg)}{u \in \Lrm^1(\Omega;\R^N)} = \R^N$. Since $\inf_{z \in \R^N} j(x,z)$ is nothing else than
	$-f^*(x,u^*(x))$ for a.e. $x \in \Omega$, it follows that
	\[
	(I_f)^*(u^*) = I_{f^*}(u^*).
	\]
\end{proof}

\begin{lemma}\label{special case}
Assume that $f: \Omega \times \R^N \to \R$ is a continuous and convex integrand with linear growth at infinity. Then,
\[
\Phi^*(u^*,w^*) =  \begin{cases}
I_{f^*}(u^* + \A^* w^*) - \dpr{w^*,\tau}& \text{if $w^* \in D(\A^*)$}\\
\infty & \hfil \text{otherwise}
\end{cases}.
\] 
\end{lemma}
\begin{proof} By definition
\begin{equation}\label{eq:diosmio}
\begin{split}
\Phi^*(u^*,w^*) & = \sup_{\substack{u \in \Lrm^1(\Omega;\R^N)\\p\in \Lrm^1(\Omega;\R^n)}} \Big\{ \dpr{u^*,u} + \dpr{w^*,p} - I_f(u) - \delta_{\Lrm^1(\Omega;\R^n)}(\A u - p \, | \{\tau\})\Big\} \\
& = \sup_{\substack{u \in \Wrm^{\A,1}(\Omega)}} \Big\{ \dpr{u^*,u} + \dpr{w^*,\A u - \tau} - I_f(u)\Big\}.
\end{split}
\end{equation}
Taking the supremum over all $u \in \Wrm^{1,\A}(\Omega)$ with $\|u\|_{\Lrm^1} \le 1$, the inequality above yields
\[
\Phi(u^*,w^*) \ge -\|u^*\|_{\Lrm^\infty} + \sup_{\substack{u \in \Wrm^{\A,1}(\Omega)\\\|u \|_{\Lrm^1} \le 1}} |\dpr{w^*,\A u}|  -\dpr{w^*,\tau}.
\]
Hence,
\[
\Phi^*(u^*,w^*) = \infty \quad \text{if $w^* \notin D(\A^*)$}.
\]
This shows the second part of the assertion.

If $w^* \in D(\A^*)$, we may use that $\dpr{w^*,\A u } = \dpr{\A^*w^*,u}$ in the last line of \eqref{eq:diosmio} to get that
\[
\Phi^*(u^*,w^*) = (I_f)^*(u^* + \A^* w^*) - \dpr{w^*,\tau}.
\]
The first part of the sought assertion is then an immediate consequence of  Lemma \ref{If}.
\end{proof}

\begin{corollary}\label{cor:nogap}The dual problem of \eqref{eq:problem} 
reads:
\begin{equation}
\tag{$\mathcal P^*$} \text{maximize} \; \mathcal R \; \text{in the space} \; \Lrm^\infty(\Omega;\R^n),
\end{equation}
where $\mathcal R: \Lrm^\infty(\Omega;\R^n) \to \cl \R$ is the functional defined as 
\[\mathcal R[w^*] := -\Phi^*(0,w^*) = \begin{cases}
\dpr{w^*,\tau}- I_{f^*}(\A^* w^*) & \text{if $w^* \in D(\A^*)$}\\
-\infty & \hfil \text{otherwise}
\end{cases}.
\] 
\end{corollary}

\begin{remark}The Lagrangian associated to the perturbation $\Phi$ is the functional $L : \Lrm^1(\Omega;\R^N) \times \Lrm^\infty(\Omega;\R^n) \to \cl \R$ given by
\[
L(u,w^*) \coloneqq I_f(u) - \dpr{w^*,\A u - \tau}.
\]
Clearly 
\[
\inf_{u \in \Lrm^1(\Omega;\R^N)} L(u,w^*) = \Rcal[w^*],
\]
and
\[
\sup_{w^* \in \Lrm^\infty(\Omega;\R^n)} L(u,w^*) = \Fcal[u].
\]
\end{remark}
Recall that, since $f$ is convex in the $z$ variable for every $x \in \Omega$, using the definition of Fenchel conjugate one obtains
\begin{align*} \mathcal F[u] & = J(u) = I_f(u) \\
& \ge \langle u,v^* \rangle  - I_{f^*}(v^*), \quad \text{for all $u \in u_0 + \Ker \A$, and every $v^* \in \Lrm^\infty(\Omega;\R^N)$}.\end{align*}
In particular, if we set $v^* = \A^* w$ with $w^* \in \Lrm^\infty(\Omega;\R^n) \cap D(\A^*)$, we get 
\[
\mathcal F[u] \ge \langle \A^* w^*, u_0 \rangle - I_{f^*}(\A w^*) = \dpr{w^*,\tau} - I_{f^*}(\A^* w^*) = \mathcal R[w^*].
\]
Hence,
\begin{equation}\label{infsup}
\inf_{u_0 + \ker \A} \mathcal F[u] \ge \sup_{\Lrm^\infty(\Omega;\R^n)} \mathcal R[w^*].
\end{equation}
So far we have not made use of the fact that $\im \A$ is a closed subspace of $\Lrm^1(\Omega;\R^N)$; however, it is precisely under this assumption that the infimum and the supremum coincide:

\begin{proof}[Proof of Theorem \ref{prop:nogap}] That problem \eqref{eq:problem} and \eqref{eq:dualproblem} are dual to each other is a consequence of Corollary \ref{cor:nogap}.

To show that the equality in \eqref{infsup} holds we argue as follows: \\
\proofstep{Step~1. Case reduction.} It is enough to look at integrands $f: \Omega \times \R^N \to \R$ with the property that
$\partial_z f$ exists for every $z \in \R^N$. Indeed, as shown in the proof of Lemma \ref{rem:rec1}, $f$ is $x$-uniformly Lipschitz in its second argument with Lipschitz constant $M$. Therefore, 
\[
|f^\delta(x,z) - f(x,z)|\le M\delta, \quad \forall \; x \in \Omega, z \in \R^N,
\] 
where the notation $f^\delta$ stands for the function 
\[
f^\delta(x,z) := (f(x,\frarg) * \rho_\delta)(z) = \int_{\R^N} f(x,y) \rho_\delta(z-y) \, \text{d} y,
\]
and where $\rho_\delta$ is a standard (smooth and even function) mollifier at scale $\delta$. It is not hard to see that the mollified $f^\delta \ge f$ are again continuous and convex integrands, and the limits: $f^\delta \downarrow f$ and $(f^\delta)^* \uparrow f^*$ hold uniformly in $\R^N$ for every $x \in \Omega$.
 Hence, thanks to \eqref{infsup} and to the monotone convergence theorem it will be sufficient to show that 
\begin{equation}\label{delta}
\inf_{u_0 + \ker \A} \mathcal F_\delta \le \sup_{\Lrm^\infty(\Omega;\R^n)} \mathcal R_{\delta}, \quad \text{for all $\delta > 0$}.
\end{equation}
where $\mathcal F_\delta$ and $\mathcal R_\delta$ are defined as $\mathcal F$ and $\mathcal R$ with $f_\delta$ and $f^*_\delta$ respectively. \\

\noindent \proofstep{Step~2. Approximative solutions.} In order to prove this inequality we will make use of the following result on approximative solutions that uses Ekeland's Variational Principle (cf. \cite[Proposition 4.2]{BeckSchmidt15}). 
\begin{lemma}
	Let $g:\Omega \times \R^N \to \R$ be a normal convex integrand with linear growth at infinity and assume that $\partial_z g(x,z)$ exists for every $(x,z) \in \Omega \times \R^N$. Let $\varepsilon$ be a positive constant and let $\overline u \in u_0 +  \ker \A$ be such that 
	\[
 I_g (\overline u) \le \inf_{u_0 + \ker \A} I_g + \varepsilon.
	\]
	Then there exist functions $\hat u \in u_0 +  \ker \A$ and $v^* \in \Lrm^\infty(\Omega;\R^N)$ with the following properties:
	\begin{gather*}
	I_g(\hat u) < \inf_{u_0 + \ker \A} I_g  + 2\varepsilon,\\
	\| \hat u - \overline u \|_{\Lrm^1}\le \sqrt{\varepsilon},\\
	v^*(x) = \partial_z g(x,\hat u(x)) \quad \textnormal{for a.e. $x \in \Omega$},\\
	\langle v^*, \eta \rangle < \sqrt{\varepsilon}\|\eta\|_{\Lrm^1} \quad \forall \; \eta \in \ker \A.
	\end{gather*}
\end{lemma}
\begin{proof}
	Regard $I_g$ as a continuous functional $I_g: (u_0 + \ker \A,{\varepsilon}^{-1/2}\| \cdot \|_{\Lrm^1}) \to \R$. It follows from  the growth conditions of $g$ that $I_g$ is well defined. Also, an application of
	Ekeland's Variational Principle tells us that there must exist $\hat u \in 
	u_0 + \ker \A$ such that
	\begin{gather*}
	\| \hat u - \overline u \|_{\Lrm^1}\le \sqrt{\varepsilon},\\
	 I_g(\hat u) < I_g(u) +\sqrt{\varepsilon}\|\hat u - u\|_{\Lrm^1} \quad \forall~u \in u_0 + \ker \A.
	\end{gather*}
	In particular, taking $u = \overline u$ it follows that
	\[
	I_g (\hat u) <I_g(\cl u) + \varepsilon \le 
	\inf_{u_0 + \ker \A} I_g + 2\varepsilon,
	\]
	this proves the first estimate.
	If this time we take $u = \hat u - s \eta$, for every given  $\eta \in \ker \A$ we get
	\[
	-\int_{\Omega }\frac{g(x,\hat u(x) - s\eta(x)) - g(x,\hat u(x))}{s} \le \sqrt{\varepsilon}\|\eta\|_{\Lrm^1},\quad \eta \in \ker \A.
	\]
	Taking the limit as $s \downarrow 0$ and using the fact that $\partial_z g(x,z)$ exists for every $(x,z) \in \Omega \times \R^N$, by Fatou's Lemma we get
	\begin{equation*}
	\langle \partial_z g(\frarg,\hat u) , \eta \rangle \le \sqrt{\varepsilon}\|\eta\|_{\Lrm^1},\quad \forall \; \eta \in \ker \A.
	\end{equation*}
	Notice that $v^* := \partial_z g(\frarg,\hat u(\frarg)) \in \Lrm^\infty(\Omega;\R^N)$ (which follows from the uniform Lipschitz bound for $g$, as it is convex and possess a uniform linear growth at infinity). This proves  the lemma.
\end{proof}
Let us go back to the proof of \eqref{delta}. Fix $\delta > 0$ and let  $(u_\varepsilon)\subset u_0 + \ker \A$ be an $\varepsilon$-minimizing sequence for $\mathcal F_{\delta}$, i.e., 
\[
\mathcal F_{\delta}[u_\varepsilon] < \inf_{u_0 + \ker \A} \mathcal F_{\delta} + \varepsilon.
\] 
It is clear that $f^\delta$ fits the requirements of the lemma above and hence we may obtain an $\Lrm^\infty$-bounded sequence $\{v_\varepsilon\} \subset \Lrm^\infty(\Omega;\R^N)$ with the following properties:
\begin{equation}\label{casi1}
(f^\delta(x,\cdot))^*(u_\varepsilon(x)) + (f^\delta(x,\cdot))^*(v_\varepsilon(x)) = u_\varepsilon(x) \cdot v_\varepsilon(x), \quad \text{for every $x \in \Omega$},
\end{equation}
and
\begin{equation}\label{casi2}
\langle v_\varepsilon, \eta \rangle \le \sqrt{\varepsilon} \|\eta\|_{\Lrm^1}, \quad \text{for every $\eta \in \ker \A$}.
\end{equation}
From the uniform boundedness of $v_\varepsilon$, we may extract a subsequence (which will not be relabeled) to find a function $v \in \Lrm^\infty(\Omega;\R^N)$ such that
$v_\varepsilon \toweakstar v$. Observe that since $\range \A$ is closed for the $\Lrm^1$-strong topology it must hold that
$\im \A^*$ is closed for the $\Lrm^\infty$ topology (see \cite[Theorem2.19]{brezis1999analyse}). Hence, $(\ker \A)^\perp = \text{Im} \A^*$ and from \eqref{casi2} we then get  that $v \in (\ker \A)^\perp = \range\A^*$. This characterization yields the existence of $w^* \in D(\A^*)$ with $v = \A^* w^*$.
It follows from the convexity of $(f^\delta)^*(x,\frarg)$ at every $x \in \Omega$, and the fact that $(f^\delta)^* : \Omega \times \R^N \to \overline \R$ is again a normal integrand bounded from below (see for example \cite[Chapter VIII]{EkelandTemamBook}) that the map $\eta \mapsto \int_{\Omega } -(f^\delta)^*(x,\eta(x)) \,\text{d}x$ is $\Lrm^\infty$-weakly* upper semicontinuous. With this and \eqref{casi1} in mind, one easily verifies that
\begin{align*}
\mathcal R_{\delta}[w^*] & :=  -\int_\Omega (f^\delta)^*(x,\A^* w^*(x)) \, \text{d}x \\
& \ge \limsup_{\varepsilon \downarrow 0}  \left\{-\int_\Omega (f^\delta)^*(x,v_\varepsilon(x)) \, \text{d}x\right\} \\
& =  \limsup_{\varepsilon \downarrow 0}   \left\{\int_{\Omega} f^\delta(x,u_\varepsilon(x)) \,\text{d}x - \langle v_\varepsilon, u_\varepsilon \rangle \right\}\\
& \ge  \lim_{\varepsilon \downarrow 0} \left\{  \Fcal_\delta[u_\eps] - \sqrt{\varepsilon} \|u_\varepsilon\|_{\Lrm^1}\right\} = \inf_{u_0 + \ker \A} \mathcal F_\delta,
\end{align*}
where in the last step we have used the coercivity of $f$ to guarantee that any minimizing sequence for $\mathcal F$ is $\Lrm^1$-uniformly bounded. This proves 
\[
\inf_{u_0 + \ker \A} \mathcal F_\delta \le \sup_{\Lrm^\infty(\Omega;\R^n)} \mathcal R_\delta.
\]
Since our choice of $\delta$ was arbitrary, this proves \eqref{delta} which in turn gives
\begin{equation}\label{finite}
\inf_{u_0 + \ker \A} \mathcal F = \sup_{\Lrm^\infty(\Omega;\R^n)} \mathcal R.
\end{equation}

To prove that there exists a solution $w^* \in \Lrm^\infty(\Omega;\R^n)$ of problem \eqref{eq:dualproblem} we observe the following facts: the set  inclusion $\set{z\in \R^N}{f^*(x,z) < + \infty} \subset M\cdot\Bbb^N$ holds 
for every $x \in \Omega$, whereby the infimum -- and hence also the supremum in \eqref{finite} -- is finite. 
We may then extract a maximizing sequence $(w_j^*) \subset D(\A^*)$ with the property that
\[
\sup_{j \in N} \|\A^* w_j^* \|_{\Lrm^\infty} \le M < \infty.
\]
The conclusion follows by the direct method: up to taking a subsequence, we may assume that $w_j^*$ weakly* converges to some $w^* \in \Lrm^\infty(\Omega;\R^N)$. Since $\im \A^*$ is closed in $\Lrm^\infty(\Omega;\R^N)$ with respect to the weak* topology, we may further find $v^* \in \Lrm^\infty(\Omega;\R^n) \cap D(\A^*)$ such that $\A^* v^* = w^*$. The conclusion is then an immediate consequence of  the sequential $\Lrm^\infty$-weakly* l.s.c.\footnote{For a normal integrand $f : \Omega \times \R^N \to \overline \R$ with linear growth at infinity, its conjugate $f^*:\Omega\times \R^N \to \overline \R$ is again a normal integrand bounded from below.} of 
\[\eta \mapsto -\int_{\Omega} f^* (x,\eta(x)) \dd x.\]

\end{proof}

\begin{remark}[Assumptions I]\label{rem:ass1} The results in this section hold for any partial differential operator  $\A$ with the property (H1); as it can be observed, property (H2) has not yet been employed in our proofs. Therefore, one can think of (H1) as a technical assumption that allows one to use convex duality methods.
\end{remark}

\section{The relaxed problem}\label{sec:relaxation}

So far we have not discussed the optimality conditions for problem \eqref{eq:problem}. In part, this owes to the fact that \eqref{eq:problem} may not necessarily be well-posed. More precisely, due to the lack of compactness of $\Lrm^1$-bounded sets one must look into
the so-called relaxation of the energy $\mathcal F$. The latter has a meaning by extending the basis space to a subspace of the bounded 
vector-valued Radon measures $\M(\Omega;\R^N)$. It is well known that the largest (below $\mathcal F$) lower semicontinuous functional with respect to the weak*-convergence of measures is given by
\[
\overline{\mathcal F}[\mu] := \inf \left\{ \liminf \mathcal F[u_j] : u_j\toweakstar \mu, u_j \in u_0 +  \ker \A \right\}.
\]
Under assumption (H2) it is easy to see that $\overline{\mathcal F}$ is again an integral functional:

\begin{proof}[Proof of Theorem \ref{prop:relax}]Let $\mu \in u_0 + \ker_{\M} \A$, we recall that 
	\[
	\overline{\mathcal F}[\mu] := \inf \setB{\liminf_{j \to \infty}\mathcal F[u_j]}{u_j \in u_0 +  \ker \A, u_j\Lcal^d \toweakstar \mu}.
	\]
	
	We divide the proof in three parts: \\
	\proofstep{1.~Lower bound.} Let $(u_j)$ be a sequence in $u_0 +  \ker \A$ with the property that 
	\[
	u_j \Lcal^d \toweakstar \mu, \qquad \text{in $\M(\Omega;\R^N)$},
	\]
	we want to show that
	\[
	\liminf_{j \to \infty}\, \mathcal F[u_j] \ge \int_{\Omega} f\left(x,\frac{\dd \mu}{\dd \Lcal^d}(x)\right) \, \textnormal{d}x + \int_{\Omega} f^\infty\left(x,\frac{\dd \mu^s}{\dd |\mu^s|}(x)\right) \, \textnormal{d}|\mu^s|(x).
	\]
	The latter is a consequence of Lemma \ref{rem:rec1} and Theorem \ref{thm:alibertlsc} (and the fact that $f \ge 0$)

	\proofstep{2.~Upper bound.}  We show that there exists a sequence $(u_j) \subset u_0 +  \ker \A$ with $u_j \Lcal^d \toweakstar \mu$ and such that
	\[
		\limsup_{j \to \infty} \Fcal[u_j]	\le  \int_{\Omega} f\left(x,\frac{\dd \mu}{\dd \Lcal^d}(x)\right) \, \textnormal{d}x + \int_{\Omega} f^\infty\left(x,\frac{\dd \mu^s}{\dd |\mu^s|}(x)\right) \, \textnormal{d}|\mu^s|(x).
		\]

This time we will make use of (H2) and Theorem \ref{thm:resch}: Let $(u_j) \subset u_0 +  \ker \A$ be a sequence that area-strict converges to $u_0 + \mu \in \ker_\M \A$.
	
	A direct consequence of Theorem \ref{thm:resch} is that
		\[
		\limsup_{j \to \infty}\, \mathcal F[u_j] = \int_{\Omega} f\left(x,\frac{\dd \mu}{\dd \Lcal^d}(x)\right) \, \textnormal{d}x + \int_{\Omega} f^\infty\left(x,\frac{\dd \mu^s}{\dd |\mu^s|}(x)\right) \, \textnormal{d}|\mu^s|(x),
		\]
		which is the sought assertion.
		
\proofstep{3.~Conclusion.} A combination of the lower and upper bounds yields that
\[
\cl\Fcal[\mu] = \int_{\Omega} f\left(x,\frac{\dd \mu}{\dd \Lcal^d}(x)\right) \, \textnormal{d}x + \int_{\Omega} f^\infty\left(x,\frac{\dd \mu^s}{\dd |\mu^s|}(x)\right) \, \textnormal{d}|\mu^s|(x),
\]
for all $\mu \in u_0 + \ker_\M \A$.
\end{proof}

\begin{remark}
The direct method can be applied to derive the existence of solutions to \eqref{eq:rproblem}. This follows from the sequential weakly* lower semicontinuity of $\cl \Fcal$ in $u_0 + \ker_\M \A$ and the coerciveness of $f$.  
\end{remark}

\section{The pairing $\llbracket \mu , \A^*u \rrbracket$ and the optimality conditions}\label{sec:pair}

The pointwise product $(\mu \cdot v^*)$ of two functions, $\mu \in u_0 + \ker \A$ and $v^* \in \Lrm^\infty(\Omega;\R^N)$, may be regarded as a bounded Radon measure through the measure  that takes the values
\[
\langle \mu,v^*\rangle (B) := \int_{B\cap \Omega} \mu(x) \cdot v^*(x) \dd x, \quad \text{$B \subset \R^N$ Borel set}.
\]
In general, if $\mu \in u_0 + \ker_\M \A$ is only assumed to be vector-valued Radon measure, one cannot simply give a notion to the inner product of $\mu$ and $v^*$ (even in the sense of distributions). However, following the interests of our minimization problem, one may define the following generalized pairing by setting
\begin{align*}
\llbracket \mu, v^*\rrbracket  :=  \setB{&\lambda \in \M(\Omega)}{ \exists \; (u_j) \subset u_0 +  \ker \A \; \text{such that} \\
& (u_j\cdot v^*)\mathcal L^d  \toweakstar \lambda \; \text{and $(u_j\Lcal^d)$  area-strict converges to $\mu$}}
\end{align*}
In this way, the set $\llbracket \mu, \A^* w^*\rrbracket$ contains information on the concentration effects of sequences of the form $(u_j \cdot \A^* w^*)$. 

The next lines are dedicated to derive the basic properties $\llbracket \mu, \A^* w^*\rrbracket$.

\begin{theorem}\label{optimal}
	Let $\mu \in u_0 + \ker_\M \A$ and let $w^* \in D(\A^*)$. Then 
		\[
	|\lambda|(\omega) \le |\mu|(\omega) \|\A^* w^*\|_\infty(\omega) \quad \text{for every Borel set $\omega \subset \Omega$},
	\]
	for all $\lambda \in \llbracket \mu,\A^*w^*\rrbracket$.
\end{theorem}
\begin{proof} Let $\lambda \in  \llbracket \mu,\A^* w^*\rrbracket$. By definition, there exists a sequence of functions $(u_j) \subset \Lrm^1(\Omega;\R^N)$ for which the measures $(u_j \Lcal^d)$ area-strict converge to $\mu$ and are such that $(u_j \cdot \A^* w^*)\Lcal^d \toweakstar \lambda$. Hence,
	\begin{equation}\label{molli}
	\liminf_{j \to \infty}|\langle u_j , \A^* w^* \rangle|(\omega) \ge |\lambda|(\omega), \quad \text{for every open set $\omega \subset \Omega$}.
	\end{equation}
	On the other hand, by H\"older's inequality, we get the upper bound
	\begin{equation}\label{molli2}
	|\langle u_j, \A^*w \rangle|(\omega) \le  |u_j|(\omega)
	\|\A^* w\|_\infty(\omega),\quad \text{for every open set $\omega \subset \Omega$},
	\end{equation}
	and every $j \in \Nbb$.
	Plugging (\ref{molli}) into (\ref{molli2}) and taking the limit as $j \to \infty$ we get, by Theorem \ref{thm:resch} (applied to $f(z) = |z|$), that
	\[
	|\lambda|(\omega) \le  |\mu|(\omega)
	\|\A^* w\|_\infty(\omega), \quad \text{for every open set $\omega \subset \Omega$ with $|\mu|(\partial \omega) = 0$}.
	\]
	The assertion for general Borel sets follows by a density argument. 
\end{proof}

\begin{corollary}\label{bracket} 	Let $\mu \in u_0 + \ker_\M \A$ and let $w^* \in D(\A^*)$. If $\lambda \in \llbracket \mu , \A^* w^*\rrbracket$, then the Radon measures measures $\lambda$ and $|\lambda|$ are absolutely continuous
	with respect to the measure $|\mu|$ in $\Omega$. Moreover,
	an application of the Radon-Nikod\`ym differentiation
	theorem yields
	$$\left\|\frac{\dd\lambda}{\dd|\mu|}\right\|_{\Lrm^\infty_{|\mu|}} \le 
	\left\|\frac{\dd|\lambda|}{\dd|\mu|}\right\|_{\Lrm^\infty_{|\mu|}} \le 
	\|\A^* w^* \|_{\Lrm^\infty}.$$
\end{corollary}
The following proposition plays a crucial role on proving the generalized saddle-point conditions; it characterizes the absolutely continuous part of elements in $\llbracket \mu,\A^* w^*\rrbracket$ and gives an upper bound for the density of its singular part. 
\begin{theorem}\label{thm:caracteriza}	Let $\mu \in u_0 +  \ker_\M \A$ and  $w^* \in D(\A^*)$. If $\lambda \in  \llbracket \mu,\A^* w^*\rrbracket$ and $\mathcal R[w^*] > - \infty$, then
	\begin{equation}\label{eq:abs2}
	\frac{\dd\lambda}{\dd|\mu|}(x) \le f^\infty\left(x,\frac{\dd\mu}{\dd|\mu^s|}(x)\right), \quad \text{for $|\mu^s|$-a.e. $x \in \Omega$},
	\end{equation}
	and
	\begin{equation}\label{eq:abs}
	\frac{\dd\lambda}{\dd\Lcal^d}(x) = \frac{\dd \lambda}{\dd \Lcal^d}(x) \cdot \A^*w^*(x), \quad \text{for $\mathcal L^d$-a.e. $x \in \Omega$}.
	\end{equation}
\end{theorem}
\begin{proof}	Let $\lambda \in \llbracket \mu, \A^*w^* \rrbracket$. By definition
	we may find sequence $(u_j) \subset \Lrm^1(\Omega;\R^N)$ that area-strict converges to $\mu$ in the sense of Radon measures, i.e., such that
	\[
	u_j \, \mathcal L^d\toweakstar \mu \in \M(\Omega;\R^N), \quad \ \area{u_j \, \mathcal L^d}(\Omega) \to \area{\mu}(\Omega),	\]
	for which
	\[
	(u_j \cdot \A^* w^*)\, \mathcal L^d \toweakstar \lambda, \qquad \text{in $\M(\Omega;\R^N)$}.
	\]
	Let $x_0 \in (\text{supp} \;\lambda^s) \cap \Omega $ be a point with the following properties:
	\begin{equation}\label{eq:Lebesgue}
	\frac{\dd\mu^s}{\dd|\mu^s|}(x_0) = \frac{\dd\mu}{\dd|\mu^s|}(x_0) = \lim_{r \downarrow 0} \frac{\mu(B_r(x_0))}{|\mu^s|(B_r(x_0))} < \infty,
	\end{equation}
	\begin{equation}
	\lim_{r \downarrow 0} \frac{\int_{B_r(x_0)} \frac{\dd \lambda}{\dd \Lcal^d}(x) \, \text{d}x}{|\mu^s|(B_r(x_0))} = 0,
	\end{equation}	
	\begin{equation}
	\frac{\dd \lambda}{\dd |\mu^s|}(x_0) = \lim_{r \downarrow 0} \frac{\lambda(B_r(x_0))}{|\mu^s|(B_r(x_0))} < \infty.\label{eq:Lebesgue2}
	\end{equation}
	Using the  principle 
	\[
	f^\infty(x,z) \ge \set{z \cdot z^*}{z^* \in \R^N,  f^*(x,z^*) < +\infty}
	\]
	and the assumption that $f^*(x,\A^* w^*(x))$ is essentially bounded by $M$ for every $x \in \Omega$ (here we use that $\Rcal[w^*] > -\infty$), we deduce the simple inequality
	\begin{equation}\label{eq:2app}
	\int_{B_s(x_0)}f^\infty(x,u_j(x)) \, \text{d}x \ge 	\int_{B_s(x_0)}u_j \cdot \A^* w^* \, \text{d}x,
	\end{equation} 
for every $s \in (0,\dist(x_0,\partial \Omega))$.
	We let $j \to \infty$ on both sides of the inequality to get  
	\begin{equation*}
	\lim_{j \to \infty} \int_{B_s(x_0)}f^\infty(x,u_j(x)) \, \text{d}x \ge\lambda(B_s(x_0)), \quad \text{for a.e. $s \in (0,\dist(x_0,\partial \Omega))$}.
	\end{equation*}
	Recall that $u_j \, \mathcal L^d$ area-strict converges to $\mu$ and by construction $f^\infty$ is positively $1$-homogeneous in its second argument. Hence, we may apply Theorem~\ref{thm:resch} to the limit in the left hand side of the inequality to obtain 
	\begin{align*}
	\frac{1}{|\mu^s|(B_s(x_0))}\int_{B_s(x_0)} & f^\infty\left(x,\frac{\dd\mu}{\dd|\mu|}(x)\right) \, \text{d}|\mu|(x) \\ & \ge \frac{\lambda(B_s(x_0))}{|\mu^s|(B_s(x_0))}, \quad \text{for a.e. $s \in (0,r)$}.
	\end{align*}
	Using properties \eqref{eq:Lebesgue}-\eqref{eq:Lebesgue2} together with the uniform Lipschitz continuity of $f^\infty$ on its second argument we deduce that
	\[
	f^\infty\left(x,\frac{\dd\mu}{\dd|\mu^s|}(x_0)\right) \ge \frac{\dd  \lambda}{\dd |\mu^s|}(x_0).
	\] 
	The sought statement follows by observing that \eqref{eq:Lebesgue}-\eqref{eq:Lebesgue2} hold simultaneously in $\Omega$ for $|\mu^s|$-a.e. $x _0 \in \Omega$.

	For the equality on Lebesgue points, let $x_0 \in \Omega$ be such that
	\begin{equation}\label{eq:nosingular}
	\lim_{r \downarrow 0} \frac{|\mu^s|(B_r(x_0))}{r^N} = 0, 
		\end{equation}
	\begin{equation}
	\qquad \lim_{r \downarrow 0} \frac{1}{r^N} \int_{B_r(x_0)} \bigg| \frac{\dd \mu}{\dd \Lcal^d}(x)  - \frac{\dd \mu}{\dd \Lcal^d}(x_0) \bigg| \dd x = 0,
		\end{equation}
	and
	\begin{equation}\label{eq:lebesguepoint}
	\frac{\dd (\frac{\dd \mu}{\dd \mathcal L^d} \cdot \A^* w^*)}{\dd\mathcal L^d}(x_0) = \frac{\dd \mu}{\dd \mathcal L^d}(x_0) \cdot \A^* w^*(x_0).
	\end{equation}
	Let 
	\[
	P_0 \coloneqq \frac{\dd \mu}{\dd \mathcal L^d}(x_0).
	\]
	Then, by definition, for a.e. $r \in (0, \dist(x_0,\partial \Omega))$ it holds that
	\begin{align*} \bigg|\lambda(B_r(x_0))  - \int_{B_r(x_0)} &  P_0 \cdot \A^*w^* \, \text{d}x\bigg| \\
	& = \lim_{n \to \infty}\left|	\int_{B_r(x_0)}u_j \cdot \A^*w^* \, \text{d}x - \int_{B_r(x_0)} P_0\cdot \A^*w^* \, \text{d}x\right| \\
	& \le \|\A^*w^*\|_{\Lrm^\infty} \cdot \lim_{j \to \infty} \int_{B_r(x_0)} |u_j - P_0 | \dd x \\
	&  \le \|\A^*w^*\|_{\Lrm^\infty} \cdot \bigg(\int_{B_r(x_0)} \bigg| \frac{\dd \mu}{\dd \Lcal^d}  - P_0\bigg| \dd x  \\ 
	& \qquad +  |\mu^s|(B_r(x_0))\bigg) = \SmallO(r^N),
	\end{align*}
	where in the last step we have used that $(u_j - P_0)\Lcal^d$ area-strict converges to $\mu - P_0 \Lcal^d$. This follows from Theorem \ref{thm:resch} and the fact that $f^\infty(\frarg - P_0) = f^\infty(\frarg)$.
	
	Essentially, this means that computing the Radon-Nikod\'ym derivative of $\lambda$ at $x_0$ is equivalent to calculate the correspondent derivative of the measure $(\frac{\dd \mu}{\dd \mathcal L^d} \cdot \A^*w^*) \mathcal L^d$ at $x_0$. Under this reasoning we use \eqref{eq:lebesguepoint} to calculate 
	\[
	\frac{\dd \lambda}{\dd \Lcal^d}(x_0) = \frac{\dd \mu}{\dd \Lcal^d}(x_0) \cdot \A^* w^*(x_0).
	\]
	Properties \eqref{eq:nosingular}-\eqref{eq:lebesguepoint} hold simultaneously for $\mathcal L^d$-a.e. $x_0 \in \Omega$ from where~\eqref{eq:abs} follows.   
\end{proof}
 
 \begin{remark}\label{rem:better}
 If $\A^* w^*$ is $|\mu^s|$-measurable, then one can prove (by a similar argument to the one used in the proof of  \eqref{eq:abs}) that
	\[
	\frac{\dd \lambda}{\dd |\mu^s|}(x_0) = \frac{\dd \mu}{\dd |\mu^s|}(x_0) \cdot \A^* w^*(x_0), \quad \text{for $|\mu^s|$-a.e. $x_0 \in \Omega$}.\]
 \end{remark}
 
 For a sequence $(u_j) \subset \Lrm(\Omega;\R^N)$ that area-strict converges to some $\mu \in \ker_{\M} \A$ it is automatic to verify, by means of Theorem \ref{thm:resch}, that 
 \begin{equation}\label{eq:como medidas}
 f(\frarg,u_j) \, \Lcal^d \toweakstar f\left(\;\frarg\;,\frac{\dd \mu}{\dd \Lcal^d}\right) \, \Lcal^d \restrict \Omega+ f^\infty\left(\;\frarg\;,\frac{\dd\mu^s}
 {\dd|\mu^s|}\right)\dd|\lambda_s|
 \end{equation}
 in $\M^+(\Omega)$. If one dispenses the assumption that $(u_j)$ area-strict converges $\mu$ and only 
 assumes that $u_j \, \mathcal L^d \toweakstar \mu$ in $\M(\Omega;\R^N)$ (or even the stronger strict convergence) the convergence \eqref{eq:como medidas} may not hold as already observed in Remark \ref{rem:sharpness}. However, as the next proposition asserts, it does hold for minimizing sequences.

\begin{theorem}[Uniqueness and improved convergence]\label{thm:minimalsequence}
	Let $(u_j) \subset u_0 +  \ker \A$ be a minimizing sequence  of problem \eqref{eq:problem} with $u_j\, \mathcal L^d \toweakstar \mu$ in $\M(\Omega;\R^N)$.
	Then the sequence of real-valued radon measures $(f(\frarg,u_j)\, \Lcal^d \restrict \Omega)$ weakly* converges in $\Omega$, in the sense of Radon measures, to the measure
	\[
	f\left(\;\frarg\;,\frac{\dd \mu}{\dd \Lcal^d}\right) \, \mathcal L^d \restrict \Omega+ f^\infty\left(\;\frarg\;,\frac{\dd\mu^s}{\dd|\mu^s|}\right)\dd |\mu^s|.
	\]	
	Even more, if $f(x,\frarg)$ and $f^\infty(x,\frarg)$ are strictly convex
	 for all $x \in \Omega$, then
	$\mu$ is the unique minimizer of \eqref{eq:rproblem} and $u_j \, \mathcal L^d$ area-strict converges to $\mu$ in $\M(\Omega;\R^N)$. 
\end{theorem}
\begin{remark}\label{rem:convex}
	Recall that strict convexity for a positively $1$-homogeneous function $g:\R^N \to \R$ -- also called strictly convex on norms -- is equivalent to the convexity of  its unit ball, this is, 
	\[
	g(z_1) = g(z_2) = g(s z_1 + (1- s)z_2), \qquad \text{ for $s \in (0,1)$}
	\]
	implies 
	\[
	z_1 = z_2.
	\]
\end{remark}

\begin{proof} Set $\Lambda_j \in \M(\Omega)$ to be the real-valued Radon measure defined as
	\[
	\Lambda_j(B) := \int_B f(x,u_j(x))\, \text{d}x, \quad \text{for any open set $B \subset \Omega$}.
	\]
	Since $(u_j)$ is a minimizing sequence, it is also $\Lrm^1$-uniformly bounded and 
	hence 
	\[
	\sup_{j \in \mathbb N}|\Lambda_j|(\Omega) < +\infty.
	\]
	We may assume, up to taking a subsequence (not re-labeled), that there exists  positive Radon measures $\Lambda,\sigma \in \M^+(\Omega)$ 
for which	\[
	\Lambda_j \toweakstar \Lambda, \quad \text{and}\quad |u_j| \Lcal^d\toweakstar \sigma.
	\]
	We do the following observation: the conclusion of Theorem~\ref{thm:alibertlsc} also holds any arbitrary open set $B \subset \Omega$. Hence,
	\[
	\Lambda(B) = \lim_{j \to \infty} \Lambda_j(B) \ge \overline{ \mathcal F}(\mu,B), 
	\]
	for every open subset $B$ of $\Omega$ with $\Lambda(\partial B) = \sigma(\partial B) = 0$, and where we have set $\overline{ \mathcal F}(\mu,\frarg)$ to be the Radon measure that takes the values
	\[
	\overline{ \mathcal F}(\mu,B) := \int_{B} f\left(x,\frac{\dd \mu}{\dd \Lcal^d}(x)\right) \, \text{d}x +
	\int_{\cl B} f^\infty\left(x,\frac{\dd \mu^s}{\dd |\mu^s|}(x)\right) \, \text{d}|\mu^s|(x),
	\]
	on open sets $B \subset \Omega$.
	Using a density argument, we conclude that
	\begin{equation}\label{eq:positivity}
	\Lambda \ge \overline{ \mathcal F}(\mu,\frarg), \quad \text{in the sense of real-valued Radon measures}. 
	\end{equation}
	So far we have not used the fact that $(u_j)$ is a minimizing sequence. Recall that, by definition, this is equivalent to 
	\[
	\Lambda_j(\Omega) \to \Lambda(\Omega) = \overline{ \mathcal F}(\mu,\Omega) = \inf_{u_0 +  \ker \A} \mathcal F.
	\]
	The latter mass convergence and \eqref{eq:positivity} are sufficient conditions for $\Lambda$ and $\overline{ \mathcal F}(\mu,\frarg)$ to represent the same Radon measure in $\M(\Omega)$, i.e.,
	\[
	\Lambda = \overline{ \mathcal F}(\mu,\frarg), \quad \text{in $\M(\Omega)$}.
	\]  
	Since the passing to a convergent subsequence was arbitrary, this proves 
	\[
 f(\frarg,u_j) \, \mathcal L^d \toweakstar f\left(\;\frarg\;,\frac{\dd \mu}{\dd \mathcal L^d}\right) \, \mathcal L^d \restrict \Omega+ f^\infty\left(\;\frarg\;,\frac{\dd\mu^s}
 {\dd|\mu^s|}\right)\dd|\lambda_s|.
	\]	
 To see that for strictly convex integrands $\mu$ is the unique minimizer of~\eqref{eq:rproblem}, one simply uses the strict convexity of $f$ and $f^\infty$, and the fact that $\ker_{\M} \A$ is a convex space. The improvement of convergence for minimizing sequences can be found in~\cite[Theorem 5.3]{AlibertBouchitte97}.  
 
\end{proof}

\begin{remark}
	The improved convergence convergence for minimizing sequences of strictly convex integrands plays no role in our characterization of the extremality conditions of problems \eqref{eq:problem} and \eqref{eq:dualproblem}. Nevertheless, we have decided to include as it is a standard result for applications in calculus of variations. 
\end{remark}

We are now in position to prove our main result: 

\begin{proof}[Proof of Theorem \ref{thm:optimal}]
	Let $\mu \in u_0 + \ker_\M \A$ be a generalized solution of problem \eqref{eq:problem}. Property (H2) tells us that there exists a sequence $(u_j) \subset u_0 + \ker \A$ that area-strict converges to $\mu$ so that $\llbracket \mu , \A^* w^*\rrbracket$ is not the empty set; we let $\lambda \in \llbracket \mu , \A^* w^*\rrbracket$.
	 
	Now let $(u_j) \subset u_0 +  \ker \A$ be an arbitrary sequence generating $\lambda$.
	By Theorem \ref{thm:resch} and the minimality of $\mu$ it also holds that $(u_j)$ is a minimizing sequence for problem \eqref{eq:problem}.
	
 In return, Theorem \ref{thm:minimalsequence} implies that the sequence of measures $(f(\frarg,u_j) \, \mathcal L^d \restrict \Omega)$ weakly* converges to the Radon measure
	\[
f\left(x,\frac{\dd\mu}{\dd\mathcal L^d}\right) \, \mathcal L^d \restrict \Omega ~+ ~ f^\infty\left(\;\frarg \;,\frac{\dd \mu}{\dd |\mu^s|}\right) \, |\mu^s|
	\]
	in $\M^+(\Omega)$.
	Since $f$ is convex on its second argument, it must hold that
	\[
	f^{**}(x,\frarg) = f(x,\frarg), \qquad \text{for every $x \in \Omega$}.\] 
	Hence,
\begin{equation}\label{eq:firstinequality}
	(f(\frarg,u_j) \, \mathcal L^d)(B) \ge  
	\int_{B}u_j \cdot \A^*w^*\, \text{d}x  - \int_{B} f^*(x,\A^* w^*)\, \text{d}x,
	\end{equation}
	for every Borel subset $B \subset \Omega$,
 herewith by Theorem \ref{thm:minimalsequence} and  \eqref{eq:firstinequality} we get (by letting $j \to \infty$) that
	\begin{align}\label{eq:thrid}
	\begin{split}
	f\left(\frarg,\frac{\dd \mu}{\dd \Lcal^d}\right) \, \Lcal^d \restrict \Omega &  + f^\infty\left(\frarg,\frac{\dd\mu}
	{\dd|\mu^s|}\right)\, |\mu^s|  \\ 
	& \ge \lambda - f^*(x,\A^* w^*) \, \mathcal L^d \restrict \Omega, 
	\end{split}
	\end{align}
	in the sense of measures. Also, by the equality in Proposition \ref{prop:nogap} we know that $\overline{\mathcal F}[\mu] = \mathcal R[w^*]$ so that
	\begin{align*}
\Bigg(f\left(\frarg,\frac{\dd \mu}{\dd \Lcal^d}\right) \,&  \Lcal^d \restrict \Omega +  f^\infty\left(\frarg,\frac{\dd\mu}
{\dd|\mu^s|}\right)\, |\mu^s| \Bigg)(\Omega)  \\ 
&  = \overline{\mathcal F}[\mu] = \mathcal R[w^*] \\
	& = \big(\dpr{w^*,\tau} -f^*(\frarg,\A^* w^*) \, \mathcal L^d \big)(\Omega) \\
	& = \big(\lambda - f^*(\frarg,\A^* w^*) \, \mathcal L^d\big)(\Omega),
	\end{align*}
	where in the last equality we used that $\lambda(\Omega) = \dpr{w^*,\tau}$ for any $\lambda \in \llbracket \mu,\A^* w^*\rrbracket$ with $\mu \in u_0 + \ker_{\M} \A$ and $w^* \in D(\A^*)$. The inequality, as measures, in \eqref{eq:thrid} and the equality of their total mass in the last formula tells us that the measures in question must be agree as elements of $\M(\Omega)$. In other words, 
	\begin{align*}
	f\left(\frarg,\frac{\dd \mu}{\dd \Lcal^d}\right) \,&  \Lcal^d \restrict \Omega +  f^\infty\left(\frarg,\frac{\dd\mu}
	{\dd|\mu^s|}\right)\, |\mu^s|  = \lambda - f^*(\frarg,\A^* w^*) \, \mathcal L^d \restrict \Omega,
	\end{align*}
	as measures in $\M(\Omega)$. 
	Finally, we recall the characterization from Theorem \ref{thm:caracteriza} so that
	\[
	f\left(x,\frac{\dd\mu}{\dd\Lcal^d}(x)\right) + f^*(x,\A^*u^*(x))  = \frac{\dd\mu}{\dd\Lcal^d}(x) \cdot \A^*w^*(x),
	\]
	for $\Lcal^d$-a.e. $x \in \Omega$, whereby
	\[
	\frac{\dd \lambda}{\dd|\mu^s|}(x) = f^\infty\left(\frac{\dd\mu}{\dd|\mu^s|}(x)\right) \quad \text{for $|\lambda^s|$-a.e. $x \in \Omega$}.
	\]
	The latter equalities fully characterize $\llbracket \mu,\A^*w^*\rrbracket$ by means of Corollary \ref{bracket} and the Radon-Nikod\'ym Decomposition Theorem. In particular, $\llbracket \mu,\A^*w^*\rrbracket$ is the singleton
		\[
		\left\{ \left(\frac{\dd\mu}{\dd\Lcal^d} \cdot \A^* w^*\right) \, \Lcal^d \restrict \Omega ~+ ~ f^\infty\left(\;\frarg \;,\frac{\dd \mu}{\dd |\mu^s|}\right) \, |\mu^s|\right\}.
		\] This proves that (i) implies (ii). 
	
	That (ii) implies (i) follows from the facts that $\inf \overline{\mathcal F} \ge \sup  \mathcal R$ and 
	\begin{equation}\label{eq:final}
	\overline{\mathcal F}[\mu] = \mathcal R[w^*].
	\end{equation}
	Indeed, the equality above implies that $\mu$ solves problem \eqref{eq:problem} and $w^*$ solves problem \eqref{eq:dualproblem}.
	To show that \eqref{eq:final} holds let $(u_j) \subset u_0 +  \ker \A$ be the (area-strict convergent) recovery sequence for $\mu$ in the proof Proposition \ref{prop:relax} so that
	\[
	f(\frarg,u_j) \toweakstar  f\left(\frarg,\frac{\dd \mu}{\dd \Lcal^d}\right) \,\Lcal^d \restrict \Omega +  f^\infty\left(\frarg,\frac{\dd\mu}
	{\dd|\mu^s|}\right)\, |\mu^s|.
	\]
	By assumption
	\[
	\lambda_j \coloneqq (u_j \cdot \A^*w^*)\Lcal^d \toweakstar \lambda \coloneqq \left(\frac{\dd\mu}{\dd\mathcal L^d}\cdot \A^* w^*\right) \, \mathcal L^d \restrict \Omega ~+ ~ f^\infty\left(\;\frarg \;,\frac{\dd \mu}{\dd |\mu^s|}\right) \, |\mu^s|
	\]
	in $\M(\Omega)$, and therefore using that $\lambda_j(\Omega) = \dpr{w^*,\tau}$ for all $j \in \Nbb$, we get that $\lambda(\Omega) = \dpr{w^*,\tau}$. The pointwise identities from (ii) then yield
		\begin{align*}
	\Rcal[w^*] & =  -\int_\Omega f^*(x,\A^*w^*) \dd x +\lambda(\Omega) \\
	&= -\int_\Omega f^*(x,\A^*w^*) \dd x +\dprB{\frac{\dd \mu}{\dd \Lcal^d},\A^*w^*} \\
	& \qquad +\int_\Omega f^\infty\left(x,\frac{\dd \mu}{\dd |\mu^s|}(x)\right) \dd  |\mu^s|(x) \\
	& = 	\cl\Fcal[\mu].
	\end{align*}
	This proves \eqref{eq:final}.
\end{proof}

\begin{remark}[Optimality conditions II]
In the case that there exists a solution $w^*$ of~\eqref{eq:dualproblem} with substantially better regularity than the one originally posed by being admissible to its variational problem, say, for example, $w \in \Crm^k(\Omega;\R^N)$ (where $k$ is the order of $\A$) or $\A^*w^* \in 
	\Lrm_{|\mu^s|}^\infty(\Omega;\R^n)$. Then, it is easy to verify (cf. Remark \ref{rem:better}) that 
	\[
f^\infty\left(x,\frac{\dd \mu^s}{\dd |\mu^s|}(x_0) \right) = \frac{\dd \mu}{\dd |\mu^s|}(x_0) \cdot \A^* w^*(x_0) \quad \text{for $|\mu^s|$-a.e. $x_0 \in \Omega$},\]
	and
	\begin{gather*}\label{pointwise}
	f\left(x,\frac{\dd\mu}{\dd\mathcal L^d}(x)\right) + f^*(x,\A^*u^*(x))  = \frac{\dd\mu}{\dd\mathcal L^d}(x) \cdot \A^*w^*(x) \quad \text{for 
		$\mathcal L^d$-a.e. in $\Omega$},
	\end{gather*}
	are also equivalent to (i) and (ii) in Theorem \ref{thm:optimal}.
\end{remark}

\bibliographystyle{plain}
\bibliography{Bib}

\begin{thebibliography}{10}

\bibitem{AlibertBouchitte97}
J.~J. Alibert and G.~{Bouchitt\'e}.
\newblock {Non-uniform integrability and generalized Young measures.}
\newblock {\em {J. Convex Anal.}}, 4(1):129--147, 1997.

\bibitem{Anzellottti83}
G.~Anzellotti.
\newblock {Pairings between measures and bounded functions and compensated
  compactness.}
\newblock {\em {Ann. Mat. Pura Appl. (4)}}, 135:293--318, 1983.

\bibitem{BeckSchmidt15}
L.~{Beck} and T.~{Schmidt}.
\newblock {Convex duality and uniqueness for {$BV$}-minimizers.}
\newblock {\em {J. Funct. Anal.}}, 268(10):3061--3107, 2015.

\bibitem{BerLas73}
H.~{Berliocchi} and J-M. {Lasry}.
\newblock {Int\'egrandes normales et mesures param\`etr\'ees en calcul des
  variations.}
\newblock {\em {Bull. Soc. Math. Fr.}}, 101:129--184, 1973.

\bibitem{BouchitteValadier88}
G.~Bouchitt{{\'e}} and M.~Valadier.
\newblock Integral representation of convex functionals on a space of measures.
\newblock {\em J. Funct. Anal.}, 80(2):398--420, 1988.

\bibitem{brezis1999analyse}
H.~Brezis, P.~G. Ciarlet, and J.~L. Lions.
\newblock {\em Analyse fonctionnelle: th{\'e}orie et applications}, volume~91.
\newblock Dunod Paris, 1999.

\bibitem{EkelandTemamBook}
I.~{Ekeland} and R.~{Temam}.
\newblock {\em {Convex analysis and variational problems.}}
\newblock pendant, 1976.

\bibitem{Kohn82(KP)}
R.~V. {Kohn}.
\newblock New integral estimates for deformations in terms of their nonlinear
  strains.
\newblock {\em Arch. Rational Mech. Anal.}, 78(2):131--172, 1982.

\bibitem{KohnTemam82}
R.~V. {Kohn} and R.~{Temam}.
\newblock {Principes variationnels duaux et th\'eor\`eme de l'\'energie dans le
  mod\`ele de plasticit\'e de Hencky.}
\newblock {\em {C. R. Acad. Sci., Paris, S\'er. I}}, 294:205--208, 1982.

\bibitem{KohnTemam83}
R.~V. Kohn and R.~Temam.
\newblock {Dual spaces of stresses and strains, with applications to Hencky
  plasticity.}
\newblock {\em {Appl. Math. Optim.}}, 10:1--35, 1983.

\bibitem{RindKris10}
J.~{Kristensen} and F.~{Rindler}.
\newblock {Relaxation of signed integral functionals in {$BV$}.}
\newblock {\em {Calc. Var. Partial Differ. Equ.}}, 37(1-2):29--62, 2010.

\bibitem{Muller87}
S.~{M{\"u}ller}.
\newblock {Homogenization of nonconvex integral functionals and cellular
  elastic materials.}
\newblock {\em {Arch. Ration. Mech. Anal.}}, 99:189--212, 1987.

\bibitem{Resh69}
Y.G. {Reshetnyak}.
\newblock {Weak convergence of completely additive vector functions on a set.}
\newblock {\em {Sib. Math. J.}}, 9:1039--1045, 1969.

\bibitem{Rockafellar}
T.~{Rockafellar}.
\newblock {Integral functionals, normal integrands and measurable selections.}
\newblock {Nonlin. Oper. Calc. Var., Summer Sch. Bruxelles 1975, Lect. Notes
  Math. 543, 157-207 (1976).}, 1976.

\bibitem{Temam83(KP)}
R.~{Temam}.
\newblock {\em Probl\`emes math\'ematiques en plasticit\'e}, volume~12 of {\em
  M\'ethodes Math\'ematiques de l'Informatique [Mathematical Methods of
  Information Science]}.
\newblock Gauthier-Villars, Montrouge, 1983.

\end{thebibliography}

\end{document}